\documentclass[12pt]{article}

\usepackage{tikz}

\usepackage{subfigure}
\usepackage[english]{babel}

\usepackage[center]{caption2}
\usepackage{amsfonts,amssymb,amsmath,latexsym,amsthm}
\usepackage{multirow}
\usepackage{mathtools}

\def\diam{\mbox{diam}}

\newtheorem{thm}{Theorem}

\newtheorem{lem}[thm]{Lemma}

\newtheorem{conj}[thm]{Conjecture}

\newtheorem{claim}{Claim}

\begin{document}

\title{The spectral radius of graphs without trees of diameter at most four\thanks{The work was supported by NNSF of China (No. 11671376) and NNSF of Anhui Province (no. 1708085MA18).}}
\author{Xinmin Hou$^a$, \quad Boyuan Liu$^b$, \quad Shicheng Wang$^c$\\
\quad Jun Gao$^d$, \quad Chenhui Lv$^e$\\
\small $^{a,b,c,d,e}$ Key Laboratory of Wu Wen-Tsun Mathematics\\
\small School of Mathematical Sciences\\
\small University of Science and Technology of China\\
\small Hefei, Anhui 230026, China.\\
}

\date{}

\maketitle

\begin{abstract}
Nikiforov (LAA, 2010) conjectured that for given integer $k$, any graph $G$ of sufficiently large order $n$ with spectral radius $\mu(G)\geq \mu(S_{n,k})$ contains all trees of order $2k+2$, unless $G=S_{n,k}$, where $S_{n,k}=K_k\vee \overline{K_{n-k}}$, the join of a complete graph of order $k$ and an empty graph of order $n-k$.
In this paper, we show that the conjecture is true for trees of diameter at most four.


\end{abstract}

\section{Introduction}
In this paper, all graphs considered are simple and finite.
For a given graph $G$,
let $A(G)$ be the adjacency matrix and let $\mu(G)$ be the largest eigenvalue of $A(G)$, we call $\mu(G)$ the {\it spectral radius} of $G$.

As Tur\'{a}n type problems ask for maximum number of edges in graphs of given order not containing a specified family of subgraphs, Brualdi-Solheid-Tur\'{a}n type problems ask for maximum spectral radius of graphs of given order not containing a specified family of subgraphs. A survey of  the Brualdi-Solheid-Tur\'{a}n type problems can be found in~\cite{survy of SEP}.

In this paper, we mainly concern a Brualdi-Solheid-Tur\'{a}n type conjecture proposed by Nikiforov~\cite{SEP of paths and cycles}. Let $S_{n,k}$ be the graph obtained by joining every vertex of a complete graph of order $k$ to every vertex of an independent set of order $n-k$, that is $S_{n,k}=K_k\vee \overline{K_{n-k}}$, the join of $K_k$ and $\overline{K_{n-k}}$, and let $S_{n,k}^+$ be the graph obtained from $S_{n,k}$ by adding a single edge to the independent set of $S_{n,k}$.

\begin{conj}[Nikiforov, 2010]\label{CONJ: conjecture 1}
Let $k\geq2$ and let $G$ be a graph of sufficiently large order $n$. If $\mu(G)\geq \mu(S_{n,k})$, then $G$ contains all trees of order $2k+2$ unless $G=S_{n,k}$.


\end{conj}

The Tur\'{a}n type version of Conjecture~\ref{CONJ: conjecture 1} is the well-known Erd\H{o}s-S\'{o}s Conjecture~\cite{E-S conj} which states that every finite simple graph with average degree greater than $k-2$  contains a copy of any tree of order $k$ as a subgraph. The Erd\H{o}s-S\'{o}s Conjecture attracts many attentions and was verified for many specific family of trees, especially for trees of diameter at most four~\cite{diameter 4}.

Nikiforov verified Conjecture~\ref{CONJ: conjecture 1} for paths~\cite{SEP of paths and cycles}.
In this paper we show that Conjecture~\ref{CONJ: conjecture 1} holds for all trees of  diameter at most four, here is our main theorem.

\begin{thm}\label{:Main}
 For $k\geq2$ and $n>{2(k+2)^4}$, every graph $G$ of order $n$ with $\mu(G)\geq \mu(S_{n,k})$ contains all trees $T$ of order $2k+2$ with $\diam(T)\le 4$ as a subgraph, unless $G=S_{n,k}$.
\end{thm}

The rest of the paper is arranged as follows. In Section 2, we give some lemmas and notation which will be used in the paper,  and  the proof of Theorem~\ref{:Main} will be given in Section 3.

\section{Lemmas and notation}

We first give some notation not defined before. Let $G=(V, E)$ be a graph. For $x\in V(G)$, define $N_G^i(x):=\{y:~y\in V(G),~d_G(x, y)=i\}$ for $ i\geq 1$, where $d_G(x,y)$ is the distance between $x$ and $y$ in $G$. In particular, $d_G(x):=|N_G^1(x)|$ is the degree of $x$ in $G$. Write $\delta(G)$ and $\Delta(G)$ for the minimum and maximum degrees of $G$, respectively. Let $\omega(G)$ and  $\diam(G)$ be the number of components and diameter of $G$, respectively. For non-empty subset $S\subseteq V$, write  $G[S]$ for the subgraph of $G$ induced by $S$ and $e(G)=|E(G)|$. For $S_1, S_2\subseteq V(G)$, let $E_G(S_1, S_2)$ be the set of edges of $G$ with one end in $S_1$ and the other in $S_2$ and write $e_G(S_1,S_2)=|E_G(S_1, S_2)|$. If $S_1$ and $S_2$ are disjoint, define  $G[S_1,S_2]$ be the  bipartite subgraph with bipartite sets $S_1$ and $S_2$ and edge set $E_G(S_1, S_2)$. The subscript $G$ will be omitted if $G$ is clear from the context. We write $[1,n]$ for the set of integers $\{1,2,\ldots, n\}$.

Write $S_k$ for a star with $k$ leaves, and we call an isolated vertex a trivial star. A vertex with the maximum degree in a star is called a center of the star. Let $S_{1,2,\ldots,2}$ be a tree obtained from a star of order $k+2$ by subdividing $k$ edges, also called a spider of order $2k+2$ with one leg of length one and the others of length two.
Write $T_\ell$ (resp. $T_{\le \ell}$) for  a tree of order $\ell$ (resp. at most $\ell$) and define

$$\mathcal{T}_{\ell}=\{ T_{\ell}\, :\,   \diam(T_{\ell})\le 4\} \mbox{ and }\mathcal{T}_{\le\ell}=\{ T_{\le\ell}\, :\,   \diam(T_{\le\ell})\le 4\},$$ and $\mathcal{T}_{2k+2}^*=\mathcal{T}_{2k+2}\setminus \{S_{1,2,\ldots,2}\}.$

Let $\mathcal{F}$ be a family of graphs. A graph $G$ is said to be $\mathcal{F}$-free if $G$ contains no member of $\mathcal{F}$ as a subgraph.
{It is well known that, either a tree $T$ has precisely one center (called {\it centered tree}), or $T$ has precisely two adjacent centers (called {\it bicentered tree})}. Furthermore, for a tree $T\in\mathcal{T}_{\le l}$, the deletion of a center of $T$ reduces a forest with each component a star.
In the following proof, a tree $T\in\mathcal{T}_{\le l}$ always is seen as a {\it rooted tree} with root at its center (for a bicentered tree, we choose one of its centers as root such that the number of components  is as large as possible in the forest obtained by deleting the root).
Let $C_T$ (resp. $F_T$) be the star forest consisting of all stars (resp. all nontrivial stars) by deleting the root of $T$. Then $F_T\subseteq C_T$ and $C_T-F_T$ consists of trivial stars (i.e. isolated vertices).

The following matrix theory lemma~\cite{Gao-Hou-17} is the theoretical base of our proof.

\begin{lem}[Lemma~6 in~\cite{Gao-Hou-17}]\label{LEM: Au-Matrix}
Given two positive integers $a, b$ and a nonnegative symmetric irreducible matrix $A$ of order $n$, let $\mu$ be the largest eigenvalue of $A$ and let $\mu'$ be the  largest root of the polynomial $f(x)=x^2-ax-b$.
Define  $B=f(A)$ and let $B_j=\sum_{i=1}^nB_{ij}$ for $j=1,2,\ldots, n$. If $B_j\leq0$ for all $j=1,2,\ldots,n$, then $\mu\leq\mu'$ with equality holds if and only if $B_j=0$ for all $j=1,2,\ldots, n$.
\end{lem}

Mclennan~\cite{diameter 4} verified Erd\H{o}s-S\'os conjecture  for trees of diameter at most four.
\begin{lem}[Theorem 1 in~\cite{diameter 4}]\label{LEM: ES-4}
Every graph $G$ with $e(G)>\frac{(k-2)|V(G)|}2$ contains a $T_{k}$ of diameter at most four.
\end{lem}

\begin{lem}[Theorem 4.1 in~\cite{matching theorem}]\label{matching theorem}
Given integers $k\geq 1$ and $n\geq\frac{5k}2$. Let $G$ be a graph of order $n$. If $G$ does not contain a matching of size $k$, then $e(G)\leq e(S_{n,k-1})$. The equality holds if and only if $G\cong S_{n,k-1}$.
\end{lem}

We also need the following variant version of Mclennan's result.

\begin{lem}\label{LEM: V-ES-4}
Given $n\ge 2k+2$. Let $G$ be a graph of order $n$ and $e(G)>\frac{(2k-1)n}2$. If $\Delta(G)=n-1$ then $G$ contains a tree $T\in \mathcal{T}^*_{\le 2k+2}$.
\end{lem}
\begin{proof}
Let $T\in \mathcal{T}^*_{\le2k+2}$.
Then each component of $C_T$ is a (trivial or nontrivial) star.

\noindent{\bf Case 1}. $C_T$ has a component, say $S_d$, of order at least three.

Let $G'$ be the graph obtained from $G$ by deleting a vertex $v$ of maximum degree and let $T'$ be the tree obtained from $T$ by  deleting $V(S_d)$. Then $$ e(G')=e(G)-(n-1)>\frac{(2k-3)(n-1)}2$$ and $|V(T')|\le 2k-1$.
By Lemma~\ref{LEM: ES-4}, $T'\subseteq G'$. Hence if we embed the center of $S_d$ in $v$, then we get an embedding of $T$ in $G$ since $d_G(v)=n-1$.

\noindent{\bf Case 2}. Each component of $C_T$ has order at most two.

Since $T\not={S}_{1,2,\ldots,2}$, then $\omega(F_T)\leq k-1$. Let $M$ be a maximum matching of $G-v$. Clearly, $m=|M|\leq k-2$, otherwise we can embed $T$ into $G$ centered at $v$ since $d_G(v)=n-1$. Assume $M= \{a_ib_i\,:\, i\in [1,m]\}$. Denote $S=V(G)\setminus (V(M)\cup\{v\})$. By the maximality of $m$, we have $e(G[S])=0$ and $e(\{a_i, b_i\}, S)\le |S|=n-2m-1$ for any $i\in [1,m]$.
Therefore,
\begin{eqnarray*}
e(G) &=& \sum_{i=1}^{m} e(\{a_i,b_i\}, S)+d_G(v)+e(G[S])+e(G[M])\nonumber\\
     &\le& m(n-2m-1)+(n-1)+m(2m-1)\nonumber\\
&=& (m+1)(n-2)-1\nonumber\\
&<&\frac{(2k-1)n}2,\nonumber\\
\end{eqnarray*}
a contradiction, where the last inequality holds since $m\le k-2$.

\end{proof}

\begin{lem}\label{LEM: Structure}
Given $n\ge 2k+2$ and let $G$ be a graph of order $n$. If $S_{n,k}$ is a subgraph of $G$, then $G$ contains a $T_{2k+2}$, unless $G=S_{n,k}$.
\end{lem}
\begin{proof}




Suppose $G\not=S_{n,k}$. Since $S_{n,k}$ is a spanning subgraph of $G$ and $n\ge 2k+2$, we have $S_{2k+2,k}^+\subseteq G$.
Let $H$ be a copy of $S_{2k+2,k}^+$ in $G$. Write $V(H)=\{x_0,x_1,\ldots,x_{2k+1}\}$ such that $H[\{x_2,\ldots,x_{k+1}\}]\cong K_k$ and $H[\{x_0, x_1,x_{k+2},\ldots, x_{2k+1}\}]\cong \overline{K_{k+2}}+{x_0x_1}$.

Let $V(T)=A\cup B$ with $|A|\leq|B|$ be the (unique) bipartition of $T$. If $|A|\leq k$, then one can embed $T$ into $H$ since $T$ is a subgraph of the complete bipartite graph $K_{|A|,2k+2-|A|}$ and $K_{|A|,2k+2-|A|}$ is a subgraph of $ H$ for any $1\leq |A|\leq k$. Now Assume that $|A|=|B|=k+1$. Since $e(T)=2k+1$, there exists at least one leaf, namely $u$, in $A$. Let $v$ be its neighbour in $B$. Then one can get an embedding of $T$ into $H$ by first embedding $u$ and $v$ into $x_0$ and $x_1$, respectively, and $A\setminus \{u\}$ and $B\setminus \{v\}$ into $\{x_2,\ldots,x_{k+1}\}$ and $\{x_{k+2},\ldots, x_{2k+1}\}$, respectively.
\end{proof}

\section{Proof of Theorem \ref{:Main}}
Note that $\mu=\mu(S_{n,k})$ is the largest root of the polynomial
\begin{equation*}
f(x)=x^2-(k-1)x-k(n-k).
\end{equation*}
Given a graph $G$ of order $n$, let $A=A(G)$ be the adjacent matrix of $G$ and let $B=f(A)$.
For $u\in V(G)$, let $B_u=\sum\limits_{1\leq i\leq n}B_{iu}$.

Now assume $G$ is a $\mathcal{T}_{2k+2}$-free graph on $n$ vertices with $\mu(G)\geq\mu(S_{n,k})$ and $G\not\cong S_{n,k}$.
By Lemma~\ref{LEM: Au-Matrix}, there must exist  a vertex $u\in V(G)$ such that $B_u\geq 0$.

\begin{claim}\label{:C1}
Every vertex $u\in V(G)$ with $B_u\ge 0$ has degree at least $k+1$.
\end{claim}
\begin{proof}[Proof of Claim 1:] For any $v\in~V(G)$, define $L=L_v$ be the graph with vertex set $V(L)=N^1(v)\cup N^2(v)$ and edge set $E(L)=E(N^1(v))\cup E(N^1(v),N^2(v))$.
By the definition of $B$, for any $v\in~V(G)$, we have
\begin{equation}\label{:C2}
 B_v=\sum\limits_{x\in N^1(v)} d_L(x)-(k-2)d_G(v)-k(n-k).
\end{equation}
Since $|N^1(v)|=d_G(v)$ and $d_L(x)\le n-2$ for any $x\in N^1(v)$, we have
\begin{equation}\label{EQN£ºe2}
B_v\leq (n-2)d_G(v)-(k-2)d_G(v)-k(n-k)=(d_G(v)-k)(n-k).
\end{equation}
If $B_v\geq0$ then $d_G(v)\ge k$. If $d_G(v)\ge k+1$ then we are done. Now assume $d_G(v)=k$. Then $B_v=0$.  By (\ref{EQN£ºe2}), we have $L\cong S_{n-1,k}$, or equivalently, $S_{n, k}$ is a subgraph of $G$. By Lemma~\ref{LEM: Structure}, either $G=S_{n,k}$ or $G$ contains a $T_{2k+2}$, a contradiction to the assumption.

 \end{proof}

Now let $u\in V(G)$ be a vertex with $B_u\ge 0$. Then $d_G(u)\geq k+1$ by Claim~\ref{:C1}. Let $T\in \mathcal{T}_{2k+2}$.
Let $p=\omega(C_T)$ and $p'=\omega(F_T)$.
Then $p'\le e(F_T)=2k+1-p$. Thus
$p'\leq k$  since $p'\le p\le 2k+1$.

\medskip

\noindent{\bf Case 1}. $d_G(u)< p$.

In this case,
we show that $T$ can be embedded in $G$ rooted at some vertex of $N^1(u)$, so we get a contradiction. For each $x\in N^1(u)$, let $$C(x)=\{y\in N^2(u)\, :\, d_L(y)\geq {2k+2-p}\}.$$

\begin{claim}\label{:C4}
There exists some vertex $x\in N^1(u)$ such that $|C(x)|\geq p'$ and $|N^1(x)\cap N^2(u)|\geq p$.
\end{claim}
\begin{proof}[Proof of Claim~\ref{:C4}:]
Suppose to the contrary that for any vertex $x\in N^1(u)$, we have $|C(x)|\leq p'-1$.
Hence
\begin{eqnarray}\label{u-e1}
&&\sum_{x\in N^1(u)}\sum_{y\in N^1(x)\cap N^2(u)}d_L(y)\nonumber\\
&&\leq |N^1(u)|\left[(p'-1)|N^1(u)|+(n-1-|N^1(u)|-p'+1)(2k+1-p)\right]\nonumber\\
&&= d_G(u)\left[(2k+1-p)n-(2k+1-p-p'+1)d_G(u)-p'(2k+1-p)\right]\nonumber\\
&&\le d_G(u)(2k+1-p)n\nonumber\\
&&\leq d_G(u)(2k-d_G(u))n\nonumber\\
&& \le (k^2-1) n,
\end{eqnarray}
the last inequality holds since $d_G(u)\ge k+1>k$.

On the other hand,
\begin{eqnarray}\label{l-e2}
&&\sum_{x\in N^1(u)}\sum_{y\in N^1(x)\cap N^2(u)}d_L(y)=\sum_{y\in N^2(u)}d_L^2(y)\nonumber\\
&&\geq\frac{1}{|N^2(u)|}\left(\sum_{y\in N^2(u)}d_L(y)\right)^2\nonumber\\
&&\ge\frac{1}{|N^2(u)|}\left(\sum_{x\in N^1(u)}d_L(x)-|N^1(u)|\left(|N^1(u)|-1\right)\right)^2\nonumber\\
&&>\frac{1}{n}\left(B_u+(k-2)d_G(u)+k(n-k)-d_G(u)(d_G(u)-1)\right)^2\nonumber\\
&&>\frac{1}{n}(kn-6k^2)^2,
\end{eqnarray}
the first inequality holds by Cauchy-Schwartz inequality, the second inequality holds since $N^2(u)$ is an independent set of $L$, and the last inequality holds since $k+1\le d_G(u)\le p\le 2k+1$. By (\ref{u-e1}) and (\ref{l-e2}), we have $\frac{(kn-6k^2)^2}{n}<(k^2-1) n$, that is $n<\frac{6k^2}{k-\sqrt{k^2-1}}<12k^3$, a contradiction to $n\ge 2(k+2)^4$. Hence, there is at least one vertex $x\in N^1(u)$ with $|C(x)|\geq p'$, without loss of generality, assume $x$ is such a vertex in $N^1(u)$ which maximizes the value of the summation $\sum\limits_{y\in N^1(x)\cap N^2(u)}d_L(y)$. Note that $d_L(y)\le |N^1(u)|=d_G(u)$. By (\ref{l-e2}), we have
\begin{eqnarray*}
|N^1(x)\cap N^2(u)|>\frac 1{d_G^2(u)}\frac{(kn-6k^2)^2}{n}\geq\frac{(kn-6k^2)^2}{9k^2n}>\frac{n}{9}-2k> p.
\end{eqnarray*}
This completes the proof of the claim.

\end{proof}

By Claim~\ref{:C4}, we can choose a vertex $x\in N^1(u)$ with $|C(x)|\geq p'$ and $|N^1(x)\cap N^2(u)|\geq p$. Note that for any vertex $y\in C(x)$, $d_L(y)\ge 2k+2-p$. Hence $d_{L-x}(y)\ge 2k+1-p$ for any $y\in C(x)$. {Therefore, we can greedily embed the star forest $F_T$  into $L-x$ with centers in $C(x)$}. 
This completes the proof of this case.


\medskip
\noindent{\bf Case 2}. $p\le d_G(u)\leq {2(k+2)^3}$.
\medskip

If we find a copy $F$ of $F_T$ in $L=L_u$ with its centers in $N^1(u)$ then the subgraph induced by $V(F)\cup \{u\}$ contains a copy of $T$ rooted at $u$,   also a contradiction to the assumption. In the following, we will show that we truly can find such an $F_T$ in $L$.

Let $C=\{x\in N^1(u)\, : \, d_L(x)\ge 2k\}$.
If $|C|\geq p'$, then we can greedily embed  $F_T$ into $L$ with its centers in $C$. Hence it is sufficient to show that such a subset $C$ exists.

\begin{claim}\label{:C3}
We have $|C|\geq k\ge p'$.
\end{claim}
\begin{proof}[Proof of Claim~\ref{:C3}:]
Suppose to the contrary that $|C|\le k-1$. Then
\begin{eqnarray*}
B_u&=&\sum_{x\in N^1(u)}d_L(x)-(k-2)d_G(u)-k(n-k)\\
&\leq&(k-1)(n-2)+(d_G(u)-k+1)\cdot 2k-(k-2)d_G(u)-k(n-k)\\
&=&-n+(k+2)d_G(u)-k^2+2\\
&\le& -n+{2(k+2)^4}-k^2+2<0,
\end{eqnarray*}
a contradiction. The proof of the case is completed.
\end{proof}

\noindent{\bf Case 3}. $d_G(u)>2(k+2)^3$.

Let $L=L_u$ be defined the same as in Claim~\ref{:C1}.
Recall that
\begin{equation*}\label{:C2}
 B_u=\sum\limits_{x\in N^1(u)} d_L(x)-(k-2)d_G(u)-k(n-k)\ge 0,
\end{equation*}
we have $\sum\limits_{x\in N^1(u)} d_L(x)\ge (k-2)d_G(u)+k(n-k)$.

If $T=S_{1,2,\ldots,2}$, then both $L$ and $G[N^1(u)]$ contain no matching of size $k$ (otherwise we have an embedding of $T$ rooted at $u$). Hence Lemma \ref{matching theorem} implies that $e(G[N^1(u)])\leq e(S_{|N^1(u)|,k-1})$ and $e(L)\leq e(S_{|N^1(u)|+|N^2(u)|,k-1})$. So we have
\begin{eqnarray*}
\sum\limits_{x\in N^1(u)} d_L(x)&=& e(G[N^1(u)])+e(L)\\
&\leq& e(S_{|N^1(u)|,k-1})+e(S_{|N^1(u)|+|N^2(u)|,k-1})\\
&=&2(d_G(u)-k+1)(k-1)+(k-1)(k-2)+(k-1)|N^2(u)|\\
&\leq&(k-2)d_G(u)+k(n-k)-|N^2(u)|.
\end{eqnarray*}
On the other hand, we have $\sum\limits_{x\in N^1(u)} d_L(x)\ge (k-2)d_G(u)+k(n-k)$. This implies that $N^2(u)=\emptyset$ and $L=G[N^1(u)]\cong S_{n-1,k-1}$, and hence $G$ contains a $S_{n,k}$ as a subgraph. Lemma \ref{LEM: Structure} implies that any tree $T_{2k+2}$ (including $S_{1,2,\ldots,2}$) is a subgraph of $G$ since $G\ne S_{n,k}$, a contradiction.

Now assume that $T\in T^*_{2k+2}$.

\medskip
\noindent{\bf Subase 3.1}. $e(N^{1}(u),N^{2}(u))>k d_{G}(u)+2k(|N^{2}(u)|+1)-k(n-k)$.

Let $G_{1}$ be the graph with vertex set $\{u\}\cup V(L)$ and edge set $E(u,N^{1}(u))\cup E(L)$. Then  we have
\begin{eqnarray*}
e(G_{1})&=&d_G(u)+\frac12\left(\sum\limits_{x\in N^1(u)} d_L(x)+e\left(N^{1}(u),N^{2}(u)\right)\right)\\
&>&d_G(u)+\frac {(k-2)d_G(u)+k(n-k)+k d_{G}(u)+2k(|N^{2}(u)|+1)-k(n-k)}2\\
&=& k(d_G(u)+|N^{2}(u)|+1)\\
&=& k|V(G_1)|.
\end{eqnarray*}
By Lemma~\ref{LEM: ES-4}, $G_1$ contains a $T_{2k+2}$ of diameter at most 4, a contradiction.

\medskip
\noindent{\bf Subcase 3.2}. $e(N^{1}(u),N^{2}(u))\le k d_{G}(u)+2k(|N^{2}(u)|+1)-k(n-k)$.

\medskip
\noindent  Let $G_{2}$ be the graph with vertex set $\{u\}\cup N^{1}(u)$ and edge set $E(u,N^{1}(u))\cup E(N^{1}(u))$. Then we have
\begin{eqnarray*}
e(G_{2})&=&d_G(u)+\frac 12\left(\sum\limits_{x\in N^1(u)} d_L(x)-|E(N^{1}(u),N^{2}(u))|\right)\\
&\geq&d_G(u)+\frac{(k-2)d_G(u)+k(n-k)-k d_{G}(u)-2k(|N^{2}(u)|+1)+k(n-k)}2\\
&=& k(n-|N^2(u)|)-k(k+1)\\
&\ge&k(d_G(u)+1)-k(k+1)\\
&>&\frac{(2k-1)(d_G(u)+1)}2\\
&=&\frac{(2k-1)|V(G_2)|}2,
\end{eqnarray*}
the last inequality holds since $d_G(u)>2(k+2)^3$.

Note that $\Delta(G_2)=d_G(u)=|V(G_2)|-1$. By Lemma~\ref{LEM: V-ES-4}, $G_2$ contains a $T_{2k+2}\in\mathcal{T}_{2k+2}^*$, a contradiction.

The proof is completed. \quad\rule{1.5mm}{2.25mm}

\section{Concluding remarks}
In fact, Nikiforov's conjecture has two parts.
\begin{conj}[Conjecture 16 in~\cite{SEP of paths and cycles}]\label{CONJ: conjecture 2}
Let $k\geq2$ and let $G$ be a graph of sufficiently large order $n$.

(a) if $\mu(G)\geq \mu(S_{n,k})$, then $G$ contains all trees of order $2k+2$ unless $G=S_{n,k}$;

(b) if $\mu(G)\geq \mu(S_{n,k}^+)$, then $G$ contains all trees of order $2k+3$ unless $G=S_{n,k}^+$.

\end{conj}
In the paper, we prove that (a) is true for all trees $T_{2k+2}\in\mathcal{T}_{2k+2}$,  and $S_{n,k}$ is the unique extremal graph with maximum spectral radius among all of the $\mathcal{T}_{2k+2}$-free graphs of order $n$.
Let $S_{2,\ldots,2}$ be the spider of order $2k+3$ with each leg of length two and let $\mathcal{T}^*_{2k+3}=\mathcal{T}_{2k+3}\setminus\{S_{2,\ldots,2}\}$.
We believe that, for sufficiently large integer $n$, (b) is true for all trees $T_{2k+3}\in\mathcal{T}_{2k+3}$, moreover,  $S_{n,k}$ is the unique extremal graph for $\mathcal{T}^*_{2k+3}$ and $S_{n,k}^+$ for $S_{2,\ldots,2}$.
We leave this as a problem.


\end{document}